%% file: main.tex
\documentclass[11pt,reqno]{amsart}
\usepackage[a4paper,textwidth=154mm,hcentering,top=28mm,bottom=28mm]{geometry}
\usepackage{amsmath,amssymb,amsthm,booktabs,hyperref}
\hypersetup{colorlinks=true,urlcolor=blue,linkcolor=blue,citecolor=blue,filecolor=blue,pdfborder={0 0 0}}
\numberwithin{equation}{section}
\newtheorem{theorem}{Theorem}[section]
\newtheorem{lemma}[theorem]{Lemma}
\newtheorem{proposition}[theorem]{Proposition}

\theoremstyle{remark}

\theoremstyle{definition}
\newtheorem{question}[theorem]{Question}
\newcommand{\E}{\mathbb E}
\newcommand{\Pp}{\mathbb P}
\newcommand{\ind}{\mathbf 1}
\title[Extremal expectations in card guessing]{Extremal expectations in card guessing with partial feedback: proofs of two conjectures of Diaconis--Graham--Spiro}
\author{Congyi Luo}
\address{School of Data Science, Fudan University, Shanghai 200433, China}
\email{cyluo24@m.fudan.edu.cn}
\date{}
\begin{document}
\begin{abstract}
We study the minimum and maximum expected scores over all adaptive strategies in sequential card guessing with yes/no feedback. A deck contains $n$ labels, each appearing $m$ times, and is shuffled uniformly. The cards are guessed one at a time, and after each guess the player is told only whether it was correct. Always guessing the same label gives exactly $m$ points. How far can the expected score depart from $m$ if all previous feedback is used either to seek or to avoid correct guesses? Diaconis, Graham, and Spiro conjectured that the minimum expectation is $m-o(m)$ as $m\to\infty$ with $n$ sufficiently large in terms of $m$, and that, for each fixed $m$, the limit superior of the maximum expectation as $n\to\infty$ is at most $(e-1)m$. Writing these extrema as $P^-_{m,n}$ and $P^+_{m,n}$, we prove
\[
 0\le m-P^-_{m,n}\le3m^{3/4}
 \qquad(n\ge m\ge1),
\]
and the finite-parameter bound
\[
 \frac{P^+_{m,n}}m\le\sum_{k=1}^{n}\frac1{k!}<e-1
\]
for all positive integers $m,n$. The first bound gives a first-order asymptotic uniform in $n\ge m$; the universal constant $e-1$ in the second cannot be reduced. Thus both conjectures follow. The lower bound combines a switching comparison for restricted permutations with martingale second-moment estimates. For the upper bound, we represent the repeated cards as a random interleaving of independent decks of distinct cards and prove that the optimal expectation is additive when the component-deck labels are revealed.
\end{abstract}
\maketitle
\input{sections/introduction}
\input{sections/model}
\input{sections/upper_single}
\input{sections/upper_bound}
\input{sections/lower_conditional}
\input{sections/lower_martingale}
\input{sections/lower_bound}
\input{sections/discussion}
\clearpage
\appendix
\section*{Appendices}
The two main proofs obtain extremal bounds by comparison, without determining an optimal strategy. Returning to exact optimization leads first to a question about the feedback history: which information must be retained? Appendix~A shows that two vectors of counts suffice, although the resulting recursion still involves many states. Appendix~B turns to decks for which exact answers are available: those in which each label occurs once, or those with only two labels. The latter case also shows that the distinction between yes/no feedback and complete feedback becomes a substantive obstacle only when a third label is present.
\input{sections/appendix_counts}
\input{sections/appendix_feedback}
\clearpage
\input{bibliography}
\end{document}

%% file: sections/introduction.tex
\section{Introduction}
Shuffle a deck with repeated labels uniformly, guess its cards one at a time, and after each guess reveal only whether it was correct. If each label occurs $m$ times, always guessing the same label gives exactly $m$ correct guesses. How far can yes/no feedback move the expected score from this baseline? An upward departure measures the player's ability to use feedback to seek correct guesses; a downward departure measures the ability to use the same information to avoid them.

Neither optimization problem is easy, even with unlimited memory and computation. An incorrect guess does not reveal the label of the card just removed, so the composition of the remaining deck is generally unknown. Moreover, a guess affects the information obtained: the label most likely to be correct now need not maximize the total expectation. Diaconis and Graham~\cite[Section~3]{DG} already exhibited this phenomenon for a deck with three labels, each appearing three times. Finding a simple optimal strategy and bounding the extrema over all strategies are different problems. We address the latter.

For a deck with $n$ labels, each appearing $m$ times, write $P^-_{m,n}$ and $P^+_{m,n}$ for the minimum and maximum expected scores. Diaconis, Graham, and Spiro~\cite[Conjectures~1 and~2]{DGS} proposed two conjectures. The first asserts that, as $m\to\infty$ with $n$ sufficiently large in terms of $m$,
\[
 P^-_{m,n}=m-o(m),
\]
with an error uniform over this range of $n$; see also~\cite[Conjecture~4.5]{DGHS}. The second asserts that, for every fixed $m$ and every $\varepsilon>0$, sufficiently large $n$ satisfies
\[
 P^+_{m,n}\le(e-1)m+\varepsilon.
\]
The first conjecture asks whether yes/no feedback is insufficient to produce an expected loss proportional to $m$, even when avoiding correct guesses is the objective. The second asks for a linear upper bound that controls all multiplicities at once; it does not predict the exact optimum for fixed $m$.

\subsection{Main results}
The two conjectures concern, respectively, a first-order asymptotic for the minimum expectation and a universal upper bound for the maximum expectation. We state these conclusions separately. Throughout, $m$ denotes the number of cards of each label and $n$ the number of labels.
\begin{theorem}[Uniform asymptotics for the minimum expectation]\label{thm:lower}
For all positive integers $m,n$ with $n\ge m$,
\begin{equation}\label{eq:lower}
 0\le m-P^-_{m,n}\le 3m^{3/4}.
\end{equation}
In particular,
\begin{equation}\label{eq:uniform-limit}
 \lim_{m\to\infty}\ \sup_{n\ge m}
 \left|\frac{P^-_{m,n}}{m}-1\right|=0.
\end{equation}
\end{theorem}
Equation~\eqref{eq:lower} directly bounds the expected loss achievable by avoiding correct guesses, while~\eqref{eq:uniform-limit} expresses the asymptotic conclusion uniformly in $n$. This proves the first conjecture, with the explicit range $n\ge m$ in place of ``$n$ sufficiently large.''
\begin{theorem}[Sharp universal upper bound for the maximum expectation]\label{thm:upper}
For all positive integers $m,n$,
\begin{equation}\label{eq:upper}
 \frac{P^+_{m,n}}{m}\le \sum_{k=1}^{n}\frac1{k!}<e-1.
\end{equation}
Equality holds in the first inequality when $m=1$ or $n=1$. The universal constant $e-1$ cannot be reduced: namely,
\begin{equation}\label{eq:sharp-constant}
 \sup_{m,n\ge1}\frac{P^+_{m,n}}{m}=e-1.
\end{equation}
\end{theorem}
Theorem~\ref{thm:upper} requires no asymptotic assumptions and hence implies the second conjecture. Equation~\eqref{eq:upper} applies directly to finite decks, and~\eqref{eq:sharp-constant} describes its sharpness over all parameters. Sharpness of the universal constant should not be interpreted as equality for a fixed $m>1$: it already follows by taking $n\to\infty$ in the classical formula $P^+_{1,n}=\sum_{k=1}^{n}1/k!$ for a single deck of distinct cards~\cite[Theorem~5]{DG}.

\subsection{Relation to earlier extremal estimates}
For large multiplicities, much finer estimates for the maximum expectation are available than $(e-1)m$. Diaconis, Graham, He, and Spiro~\cite[Theorem~1.3]{DGHS} proved that, for $n$ sufficiently large in terms of $m$,
\[
 P^+_{m,n}\le m+O\bigl(m^{3/4}(\log m)^{1/4}\bigr).
\]
Nie~\cite[Theorem~1.3]{Nie} subsequently proved
\[
 P^+_{m,n}\le m+500\sqrt m\qquad(n\ge1200\sqrt m).
\]
Our upper bound therefore determines the universal linear constant over all parameters, particularly for fixed small multiplicities, rather than the finer order of the error for large multiplicities. On the other hand, Diaconis, Graham, and Spiro~\cite[Theorem~1]{DGS} constructed simple strategies showing that
\[
 P^-_{m,n}\le m-\frac1{40}\sqrt m,
 \qquad
 P^+_{m,n}\ge m+\frac1{40}\sqrt m
 \qquad(n\ge8m).
\]
These strategies show that the effect of feedback does not disappear altogether. Our lower bound shows that it disappears from the first-order asymptotic of the minimum expectation.

Another easily described strategy is to guess 1 until correct, then 2 until correct, and so on. If all $n$ labels have been guessed correctly in this order, continue guessing $n$ for the remaining rounds. Let $R$ be the length of the longest subsequence of the form $1,2,\ldots,k$ in the card sequence. Choosing the earliest available position at each step cannot obstruct a later choice, so the increasing part of this strategy produces $R$ points. Further correct guesses may occur after all $n$ labels have been completed; thus its total score is at least $R$. Clifton et al.~\cite[Theorem~1.1]{CDHSY} determined the limit of $\E R$ as $n\to\infty$ for fixed $m$. This gives a lower bound for $P^+_{m,n}$, but analyzing one strategy cannot bound all strategies from above.

\subsection{Ideas of the two proofs}
The upper bound follows from a product representation of a uniform permutation. Distribute the $m$ cards of each label among $m$ component decks, each containing $n$ distinct labels. Shuffle these decks independently and interleave them randomly. Erasing the component-deck labels gives exactly the original uniform card sequence. If the player is told in advance which deck occupies each position, the maximum expectation can only increase. Once the other decks and the interleaving order are fixed, any strategy that uses feedback across decks induces a legal single-deck strategy on a specified deck. The optimal expectation with the deck labels revealed is therefore exactly the sum of the $m$ single-deck optima. We construct the induced strategy round by round to justify this comparison for adaptive guesses.

The lower bound uses a different symmetry left by the feedback history. After the correctly guessed positions are deleted, each failed position forbids just one label, while the future positions remain interchangeable. Switching these two kinds of positions gives a lower bound for the next success probability without solving for the entire conditional distribution. We then normalize the score deviation for each label by its number of guesses, so that its increments match this probability comparison. The weights needed to recover the total score depend on the final feedback. We retain the cancellation provided by these weights before applying Cauchy--Schwarz to control their correlation with the terminal martingale values. The sum of the second moments over all coordinates is bounded in terms of the expected total number of successes, which is at most $m$ for a minimizing strategy. Thus we need neither characterize that strategy nor obtain a precise upper bound on success probabilities for all strategies.

Section~\ref{sec:model} defines the model and the strategy class. Sections~\ref{sec:maximum} and~\ref{sec:minimum} prove the upper bound for the maximum expectation and the lower bound for the minimum expectation, respectively. The final section discusses the limit for fixed multiplicity and the scale of the loss for large multiplicity. The appendices give exact counts for feedback histories, two solvable special cases, and comparisons with complete feedback on the same deck.

%% file: sections/model.tex
\section{Model and strategies}\label{sec:model}
We define the strategy class used in the introduction and show that both extrema are attained by deterministic strategies. It will therefore suffice to prove the main theorems for such strategies. This reduction imposes no restriction on the use of previous feedback.

Fix positive integers $m,n$. Let $[n]=\{1,\ldots,n\}$, and let $\Omega_{m,n}$ be the set of words of length $mn$ over $[n]$ in which each label occurs exactly $m$ times. The card sequence $X=(X_1,\ldots,X_{mn})$ is uniform on $\Omega_{m,n}$. In round $t$, the player first guesses $G_t\in[n]$ and then receives only the feedback $Y_t=\ind\{G_t=X_t\}$. A guess may depend on all previous guesses and feedback, and on a random seed independent of the card sequence. There are no restrictions on memory or computation. For such a strategy $A$, define
\[
 S=\sum_{t=1}^{mn}Y_t,\qquad
 P^-_{m,n}=\inf_A\E_A S,\qquad P^+_{m,n}=\sup_A\E_A S.
\]
A deterministic strategy is a family of functions $A_t:\{0,1\}^{t-1}\to[n]$, with $G_t=A_t(Y_1,\ldots,Y_{t-1})$. Earlier guesses are recursively determined by the same feedback history and need not be separate arguments. After specifying values on every history, including histories of probability zero, there are only finitely many deterministic strategies.

For a randomized strategy, draw a seed $U$ independent of the card sequence before play begins. Fixing $U=u$ gives a deterministic strategy $A^u$. Since the score lies between 0 and $mn$,
\[
 \mathbb E_A S=\int\mathbb E_{A^u}S\,d\mathbb P_U(u).
\]
Randomization therefore gives only convex combinations of deterministic expectations, and both extrema are attained by deterministic strategies. It also follows that bounds valid uniformly for all deterministic strategies apply to randomized strategies. Always guessing one label gives exactly $m$ successes, so $P^-_{m,n}\le m\le P^+_{m,n}$. All logarithms are natural.

To simplify repeated expressions in the upper-bound proof, write
\[
 v_n=P^+_{1,n}=\sum_{k=1}^{n}\frac1{k!}
\]
for the optimal expectation on a single deck of distinct cards. The finite-parameter bound in Theorem~\ref{thm:upper} then becomes $P^+_{m,n}\le mv_n$.

For $m\le81$, the lower bound in Theorem~\ref{thm:lower} is no stronger than nonnegativity of the score. Small multiplicities require separate analysis. Appendix~\ref{app:feedback} gives two boundary cases with exact solutions.

We prove the upper bound for the maximum expectation first, and then the lower bound for the minimum expectation.

%% file: sections/upper_single.tex
\section{Maximum expectation}\label{sec:maximum}
We prove Theorem~\ref{thm:upper}: the expected score of any strategy is at most the sum of $m$ single-deck optima. We first represent the repeated cards as a random interleaving of independent component decks. We then reveal the deck labels and show that the optimal expectation with this additional information is exactly $mv_n$.

\subsection{The single-deck optimum}\label{sec:single}
Diaconis and Graham~\cite[Theorem~5]{DG} determined an optimal strategy and its expectation for distinct cards. This exact value will serve as the comparison for repeated cards.
\begin{lemma}\label{up:lem:single}
For a uniform permutation of $n$ distinct cards, the optimal expectation with yes/no feedback is $P^+_{1,n}=v_n$. One optimal strategy guesses 1 until correct, then 2 until correct, and continues in this way.
\end{lemma}
\begin{proof}
Optimality follows from~\cite[Theorem~5]{DG}, whose assumptions are precisely distinct cards, a uniform permutation, and only yes/no feedback after each round. We compute the value here to explain the factorial sum. Let $R$ be the score of the stated strategy. For $1\le k\le n$, the event $\{R\ge k\}$ occurs exactly when the labels $1,\ldots,k$ appear in increasing relative order in the deck. Their $k!$ relative orders are equally likely in a uniform permutation, so
\[
 \Pp(R\ge k)=\frac1{k!},\qquad
 \E R=\sum_{k=1}^n\Pp(R\ge k)=v_n.
\]
The last equality uses the pathwise identity $R=\sum_{k=1}^n\ind\{R\ge k\}$. The randomization reduction in Section~\ref{sec:model} gives the same optimum for randomized strategies.
\end{proof}

\subsection{A product representation with hidden component decks}
To apply the single-deck result to repeated cards, we must preserve both uniformity of the card sequence and independence of the component decks. Once cards of the same label are distinguished by deck labels, a uniform permutation of the distinguishable cards decomposes into independent deck permutations and an interleaving schedule.

Write each card's identity as $(r,i)\in[m]\times[n]$, where $r$ is its deck label. Let $Z$ be a uniform random permutation of these $mn$ distinct cards. Erase $r$, retaining only $i$, and write $X(Z)$ for the resulting word. Let $\mathfrak S_n$ denote the permutations of $[n]$, and let $\mathcal L_{m,n}$ be the set of sequences of length $mn$ in which each deck label $r\in[m]$ appears exactly $n$ times.
\begin{lemma}\label{up:lem:product}
The word $X(Z)$ is uniform on $\Omega_{m,n}$. Let $L_t$ be the deck label at position $t$ of $Z$, and let $\Pi_r\in\mathfrak S_n$ record the labels from deck $r$ in their order of appearance. Then
\[
 L,\Pi_1,\ldots,\Pi_m
\]
are mutually independent. Each $\Pi_r$ is uniform on $\mathfrak S_n$, and $L$ is uniform on the set $\mathcal L_{m,n}$ of schedules in which every deck label appears exactly $n$ times.
\end{lemma}
\begin{proof}
Fix $x\in\Omega_{m,n}$. For each label $i$, there are $m!$ ways to assign distinct deck labels to its $m$ occurrences. Exactly $(m!)^n$ distinguishable-card permutations therefore map to $x$. This number is independent of $x$, proving the first assertion.

The map
\[
 Z\longmapsto (L,\Pi_1,\ldots,\Pi_m)
\]
is a bijection from distinguishable-card permutations to $\mathcal L_{m,n}\times(\mathfrak S_n)^m$. Given the right-hand side, take the next label from the specified deck permutation at each step of $L$ to reconstruct $Z$ uniquely. Thus the uniform measure becomes the uniform measure on this finite Cartesian product, which is the product of the uniform measures on its factors. Explicitly,
\[
 |\mathcal L_{m,n}|\,(n!)^m
 =\frac{(mn)!}{(n!)^m}(n!)^m=(mn)!.
\]
For every $\ell\in\mathcal L_{m,n}$ and $\pi_1,\ldots,\pi_m\in\mathfrak S_n$,
\[
 \mathbb P(L=\ell,\Pi_1=\pi_1,\ldots,\Pi_m=\pi_m)
 =\frac1{(mn)!}=\frac1{|\mathcal L_{m,n}|}\prod_{r=1}^m\frac1{n!},
\]
which gives the claimed independence and uniform marginals.
\end{proof}

%% file: sections/upper_bound.tex
\subsection{The optimum with the schedule revealed}\label{sec:enhanced}
When deck labels are revealed, applying an optimal single-deck strategy separately on each deck gives expected score $mv_n$. We show that no strategy can do better, even if its guesses depend jointly on the past feedback from every deck. Once the other decks and the interleaving order are fixed, its guesses on a specified deck still form a single-deck strategy with only yes/no feedback.

Consider an enhanced experiment: sample as in Lemma~\ref{up:lem:product}, and reveal the entire schedule $L$ before the first guess. Card labels remain hidden, and each round still gives only yes/no feedback. Denote the optimal expectation by $\widetilde P^+_{m,n}$.
\begin{lemma}\label{up:lem:additive}
The enhanced experiment satisfies
\[
 \widetilde P^+_{m,n}=m v_n.
\]
\end{lemma}
\begin{proof}
We first prove the upper bound. Since the schedule takes values in a finite set, we may again fix a deterministic strategy $A$ and subsequently mix over randomized strategies. Let $S_r$ be the number of successes on deck $r$, so $S=\sum_rS_r$.

Fix $r\in[m]$. We fix the other decks to construct a single-deck strategy; their card labels are not additionally revealed to the player in the original experiment. For any fixed external data
\[
 w=\bigl(\ell,(\pi_s)_{s\ne r}\bigr),
\]
construct a strategy $B_w$ for $n$ distinct cards. Determine the feedback seen by $A$ recursively for $t=1,\ldots,mn$:
\begin{enumerate}
\item If $\ell_t=s\ne r$, the permutation $\pi_s$ is fixed, so the card at this round is known. Compare it with the guess prescribed by $A$ from the previous feedback to determine the current feedback.
\item If $\ell_t=r$, use the label prescribed by $A$ as the next guess in the single-deck game, and use its actual yes/no outcome as the feedback at round $t$.
\end{enumerate}
Before the $k$th card of deck $r$ is dealt, this recursion uses only the fixed data $w$ and the preceding $k-1$ feedback bits from that deck. The $k$th guess is therefore a function of those bits, so $B_w$ is a strategy allowed by Lemma~\ref{up:lem:single}. Guesses on the other decks can also depend on feedback from deck $r$, but the recursion proceeds in time order and uses no feedback that has not yet been obtained.

Now fix the permutation $\pi_r$. If the feedback through round $t-1$ agrees with the original game, then $A$ prescribes the same guess at round $t$. When $\ell_t\ne r$, the recursion uses the fixed card from the original game; when $\ell_t=r$, it uses the same card of $\pi_r$. The feedback at round $t$ therefore agrees as well. By induction, the score of $B_w$ on $\pi_r$ is exactly $S_r$, the original game's score on deck $r$.

By Lemma~\ref{up:lem:product}, the permutation of deck $r$ is independent of the external data $W_r=(L,(\Pi_s)_{s\ne r})$. Conditional on those data, deck $r$ remains uniform, and the single-deck optimum gives
\[
 \E[S_r\mid W_r]\le v_n
\]
for every possible $W_r$. Taking expectations and summing over $r$ gives $\mathbb E_A S\le mv_n$. Maximizing over all strategies yields $\widetilde P^+_{m,n}\le mv_n$.

For the reverse inequality, apply the optimal strategy of Lemma~\ref{up:lem:single} separately to each deck. The revealed schedule tells the player which deck is active, and the strategy state can be maintained for each deck separately. Each deck contributes expectation $v_n$, giving total expectation $mv_n$.
\end{proof}
\begin{proof}[Proof of Theorem~\ref{thm:upper}]
Any strategy for the original experiment can ignore $L$ in the enhanced experiment. By Lemma~\ref{up:lem:product}, erasing the deck identities gives exactly the original distribution, so the strategy has the same expected score in both experiments. Hence
\[
 P^+_{m,n}\le\widetilde P^+_{m,n}=mv_n.
\]
Since $v_n$ is a strict partial sum of $\sum_{k\ge1}1/k!=e-1$, this proves~\eqref{eq:upper}.

Equality for $m=1$ follows from Lemma~\ref{up:lem:single}. For $n=1$, every card has the same label, and $P^+_{m,1}=m=mv_1$. Finally, if $P^+_{m,n}\le Cm$ for all parameters, setting $m=1$ and letting $n\to\infty$ gives $C\ge e-1$.
\end{proof}
Theorem~\ref{thm:upper} directly implies Conjecture~2 of~\cite{DGS}, with a bound valid for every finite $n$.

%% file: sections/lower_conditional.tex
\section{Minimum expectation}\label{sec:minimum}
We prove Theorem~\ref{thm:lower}: when $n\ge m$, no strategy can lower the expected score below $m-3m^{3/4}$. We first use the restrictions imposed by feedback to bound the conditional success probability from below. Martingale second moments then control the deviations caused by adaptive guesses, and these estimates combine to bound the expected loss.

\subsection{Feedback histories and conditional distributions}\label{sec:conditional}
The difficulty is that a player may repeatedly choose the same label in response to earlier feedback. The next success probability therefore cannot be replaced by the unconditional probability $1/n$. We first compare the marginals conditional on a history, and then estimate how much deviation adaptive choices can accumulate.

By the reduction in Section~\ref{sec:model}, fix a deterministic strategy. Let $N=mn$, $S_t=\sum_{u=1}^tY_u$, $S_0=0$, and $\mathcal F_t=\sigma(Y_1,\ldots,Y_t)$. Since guesses are determined by past feedback, $G_t$ is $\mathcal F_{t-1}$-measurable. For $i\in[n]$, define
\[
 a_i(t)=\sum_{u=1}^t\ind\{G_u=i\},\qquad
 s_i(t)=\sum_{u=1}^t\ind\{G_u=i\}Y_u.
\]
These count the guesses of label $i$ and the successful ones, respectively, and satisfy
\[
 \sum_i a_i(t)=t,\qquad \sum_i s_i(t)=S_t,\qquad
 0\le s_i(t)\le m.
\]

\subsection{A comparison for restricted permutations}
A failure excludes only one label at its position, while positions not yet guessed are not directly restricted by the history. If all positions allowing $i$ were symmetric, their probability of containing $i$ would equal the number of remaining copies of $i$ divided by the number of allowed positions. We show that restrictions on other labels cannot make the probability at a free position much smaller than this ratio.

Consider words of length $M$ over $[n]$ in which label $i$ occurs exactly $r_i$ times, where $0\le r_i\le m$ and $\sum_i r_i=M$. Specify some positions, each forbidding one label. The other $L\ge1$ positions are unrestricted and are called free. Let $\mathcal W$ be the set of words satisfying these restrictions, assume $\mathcal W\ne\varnothing$, and give it the uniform measure. Let $f_i$ count the positions forbidding $i$, and let $p_i$ be the probability that a specified free position contains $i$. Permuting the free positions preserves $\mathcal W$, so $p_i$ does not depend on the chosen position.
\begin{lemma}\label{low:lem:switch}
If $L\ge m$, then for every $i\in[n]$,
\[
 p_i\ge\left(1-\frac mL\right)\frac{r_i}{M-f_i}.
\]
\end{lemma}
\begin{proof}
Every free position allows $i$, so $M-f_i\ge L>0$. If $L=m$, the right-hand side is zero and the result follows. Assume henceforth that $L>m$, and fix a free position $z$.

Let $x$ be a position forbidding $j$, where $j\ne i$. In $\mathcal W$, consider
\[
 A_x=\{w_x=i,\ w_z\ne j\},\qquad B_z=\{w_z=i\}.
\]
Swapping positions $x,z$ is a bijection from $A_x$ to $B_z$. Indeed, for $w\in A_x$, the new label at $x$ is not $j$, and the new label at $z$ is $i$. All other positions and all multiplicities are unchanged, so the resulting word belongs to $\mathcal W$ and to $B_z$. Conversely, for $w\in B_z$, its original label $w_x$ is not $j$. After the swap, $x$ contains the allowed label $i$ and $z$ contains the original $w_x\ne j$, so the word belongs to $A_x$. Applying the swap twice restores the word. Uniformity gives
\[
 p_i=\mathbb P(A_x).
\]
If $\mathbb P(w_x=i)=0$, the inequality $p_i\ge(1-m/L)\mathbb P(w_x=i)$ is immediate. Otherwise, conditional on $w_x=i$, the free positions are still symmetric, and at most $r_j$ of them contain $j$. Hence
\[
 L\mathbb P(w_z=j\mid w_x=i)
 =\mathbb E\left[\sum_{y\text{ free}}\ind\{w_y=j\}\,\middle|\,w_x=i\right]
 \le r_j\le m.
\]
Thus
\[
 p_i=\mathbb P(w_x=i)
       \bigl(1-\mathbb P(w_z=j\mid w_x=i)\bigr)
 \ge(1-m/L)\mathbb P(w_x=i).
\]
If $x$ is itself free, then $\mathbb P(w_x=i)=p_i$ and the same comparison holds. If $x$ forbids $i$, that probability is zero. Let $I_i$ be the set of positions allowing $i$. Summing the comparison over $x\in I_i$ gives
\[
 (M-f_i)p_i
 \ge(1-m/L)\sum_{x\in I_i}\mathbb P(w_x=i)
 =(1-m/L)r_i.
\]
The last equality holds because each word contains exactly $r_i$ copies of $i$, all in $I_i$. Division by $M-f_i>0$ proves the claim.
\end{proof}

\subsection{A probability bound conditional on feedback}
The switching lemma applies to every feedback history that can occur. After successful positions are deleted, only the past failed positions and the future free positions remain. The resulting lower bound involves only the numbers of guesses and successes for the current label.

Fix $t\ge1$ and a positive-probability feedback history $(y_1,\ldots,y_{t-1})$. For a deterministic strategy, this history recursively specifies every earlier guess. The compatible card sequences are exactly those satisfying $X_u=G_u$ when $y_u=1$ and $X_u\ne G_u$ when $y_u=0$. Sufficiency follows by induction on $u=1,\ldots,t-1$: if the first $u-1$ feedback bits agree, the strategy makes the same guess in round $u$, and the position restriction makes the next feedback bit agree as well. Thus the history event imposes no further unstated restrictions. Since the original sequence is uniform, its conditional law is uniform on the compatible sequences.

Deleting the successful positions is a bijection; its inverse reinserts the known labels at the known positions. Abbreviate $a_i=a_i(t-1)$ and $s_i=s_i(t-1)$ in this paragraph. After deletion, the parameters are
\[
 M=N-S_{t-1},\quad r_i=m-s_i,\quad f_i=a_i-s_i,
 \quad L=N-t+1.
\]
Each past failed position forbids one label, and the $L$ future positions are free. When $L\ge m$, Lemma~\ref{low:lem:switch} therefore applies to every positive-probability history.

Near the end of the deck, fewer free positions remain and the comparison weakens. To retain a uniform comparison factor, discard the last $B$ rounds and study the first $T=N-B$ rounds. Take an integer $m\le B<N$ and set $\eta=m/B$. For $1\le t\le T$, define
\[
 p_t=\mathbb E(Y_t\mid\mathcal F_{t-1}),\qquad
 q_t=\frac{m-s_{G_t}(t-1)}{N-a_{G_t}(t-1)}.
\]
Here $p_t$ is the actual conditional success probability, whereas $q_t$ is only a comparison ratio. Its numerator $m-s_{G_t}(t-1)$ includes cards at past failed positions and is not the number actually remaining in the deck. We have $N-a_{G_t}(t-1)\ge N-(t-1)\ge B+1>0$ and $0\le q_t\le m/(B+1)\le1$. For a given history, set $i=G_t$. Since $L=N-t+1\ge B+1>m$, the switching lemma applies. Moreover,
\[
 0<M-f_i=N-a_i-S_{t-1}+s_i\le N-a_i,
\]
so
\begin{equation}\label{low:eq:lower}
 p_t\ge\left(1-\frac mL\right)\frac{m-s_i}{M-f_i}
 \ge(1-\eta)q_t.
\end{equation}
The second inequality uses $m-s_i\ge0$ and $1-\eta\ge0$, so comparing denominators does not reverse the inequality.

We must now accumulate these one-round comparisons into a bound for the total score. Summing $q_t$ directly is inconvenient because both its numerator and its denominator depend on the strategy. The following normalization makes these changes cancel.

%% file: sections/lower_martingale.tex
\subsection{Martingale control of adaptive guesses}\label{sec:martingale}
We seek to control the deviation of the total score from $t/n$. For a single label, the immediate choice is $s_i(t)-a_i(t)/n$, but its mean increment involves $p_t-1/n$, whereas~\eqref{low:eq:lower} compares $p_t$ with $q_t$. The normalization replaces $1/n$ in the increment by $q_t$.

Specifically, if there have been $a$ guesses of $i$ and $s$ successes before guessing $i$ again, then
\[
 \frac{m-s}{N-a}-\frac mN
 =-\frac{s-(m/N)a}{N-a}.
\]
The departure of the comparison ratio from $m/N$ is thus determined by the score deviation already accumulated. Dividing that deviation by $1-a/N$ cancels this term as the guess count changes. For $i\in[n]$ and $0\le t\le T$, set
\[
 D_i(t)=\frac{s_i(t)-(m/N)a_i(t)}{1-a_i(t)/N},
\]
and, for $1\le t\le T$, write
\[
 c_t=\frac{N}{N-a_{G_t}(t)}.
\]
Since $a_i(t)\le t\le N-B$, the two denominators are at least $B/N$ and $B$, respectively, and hence positive. Although $c_t$ is written using the guess count at time $t$, it is determined before the feedback arrives: $a_{G_t}(t)=a_{G_t}(t-1)+1$ is $\mathcal F_{t-1}$-measurable. Also $c_t\le N/B$.

When $G_t=i$, abbreviate $a=a_i(t-1)$, $s=s_i(t-1)$, and $y=Y_t$. Putting the fractions over a common denominator gives
\begin{align*}
 D_i(t)-D_i(t-1)
 &=\frac{N(s+y)-m(a+1)}{N-a-1}-\frac{Ns-ma}{N-a}\\
 &=\frac{N\{(N-a)y-(m-s)\}}{(N-a-1)(N-a)}
 =c_t(y-q_t).
\end{align*}
When $G_t\ne i$, this coordinate does not change. It remains to estimate separately the mean-zero increment $Y_t-p_t$ and the comparison error $p_t-q_t$. Define
\begin{align*}
 M_i(t)&=\sum_{u=1}^t\ind\{G_u=i\}c_u(Y_u-p_u),\\
 A_i(t)&=\sum_{u=1}^t\ind\{G_u=i\}c_u(p_u-q_u),
\end{align*}
with $M_i(0)=A_i(0)=0$. Summing the increment identity gives $D_i(t)=M_i(t)+A_i(t)$. All summands are bounded and hence square-integrable. Since
\[
 \mathbb E\bigl[\ind\{G_t=i\}c_t(Y_t-p_t)\mid\mathcal F_{t-1}\bigr]=0,
\]
each $M_i$ is a square-integrable martingale starting at zero.

To recover the total score, we multiply back by the normalization denominators. Set
\[
 b_i=\frac{a_i(T)}N,\qquad w_i=1-b_i.
\]
The identities $\sum_i a_i(T)=T$ and $\sum_i s_i(T)=S_T$ give, pathwise,
\begin{equation}\label{low:eq:reconstruction}
 S_T-\frac{mT}{N}
 =\sum_i w_iD_i(T)
 =\sum_i w_iM_i(T)+\sum_i w_iA_i(T).
\end{equation}
A lower bound for the score thus reduces to lower bounds for the two weighted sums. The weights $w_i$ depend on the final feedback, so the first sum need not have expectation zero merely because the martingales do. We retain this dependence and treat the drift and martingale terms separately.

\subsubsection{A lower bound for the drift}
Each negative drift increment is controlled by $\eta q_t$. Multiplication by the weights that recover the score makes adjacent denominators telescope, leaving an accumulated error that depends only on the total number of guesses.

By~\eqref{low:eq:lower}, $p_t-q_t\ge-\eta q_t$. The terminal weights depend on the feedback history, but pathwise $0\le b_i,w_i\le1$ and $\sum_i b_i=T/N$. At successive guesses of $i$, its previous guess count takes the values $0,1,\ldots,a_i(T)-1$. Retaining $w_i=(N-a_i(T))/N$ and using $m-s_i(t-1)\le m$, we obtain
\begin{align*}
 w_i\sum_{t\le T:G_t=i}c_tq_t
 &\le w_i mN\sum_{a=0}^{a_i(T)-1}
 \frac1{(N-a-1)(N-a)}\\
 &=w_i mN\left(\frac1{N-a_i(T)}-\frac1N\right)
 =\frac{m a_i(T)}N.
\end{align*}
The sum telescopes by taking differences of adjacent reciprocals; both sides are zero if $a_i(T)=0$. Consequently,
\begin{equation}\label{low:eq:drift}
 \sum_iw_iA_i(T)
 \ge-\eta\sum_iw_i\sum_{t\le T:G_t=i}c_tq_t
 \ge-\frac{m^2T}{NB}.
\end{equation}
This estimate holds for every feedback history. The cancellation between the normalization denominator and the terminal weight avoids any need to estimate the number of guesses of each label separately.

\subsubsection{Second moments with random weights}
Summing the second moments over all coordinates leaves one weighted conditional variance per round. After applying $c_t\le N/B$, the expected sum of variances is at most $\E S_T$. This is why we can ultimately use $\E S_T\le m$ only for a minimizing strategy. A pathwise norm bound on the weights then handles their correlation with the terminal martingale values.

Let $\Delta M_i(t)=M_i(t)-M_i(t-1)$. If $u<t$, then $\Delta M_i(u)$ is $\mathcal F_{t-1}$-measurable, and hence
\[
 \mathbb E[\Delta M_i(u)\Delta M_i(t)]
 =\mathbb E\bigl[\Delta M_i(u)\mathbb E(\Delta M_i(t)\mid\mathcal F_{t-1})\bigr]=0.
\]
Expanding the square, eliminating the cross terms, and using the conditional Bernoulli law of $Y_t$, we obtain
\begin{align*}
 \sum_i\mathbb E M_i(T)^2
 &=\sum_{i,t}\mathbb E\bigl[\ind\{G_t=i\}c_t^2(Y_t-p_t)^2\bigr]\\
 &=\mathbb E\sum_{t=1}^Tc_t^2p_t(1-p_t)
 \le\left(\frac NB\right)^2\mathbb E\sum_{t=1}^Tp_t
 =\left(\frac NB\right)^2\mathbb E S_T.
\end{align*}
For each $t$, precisely one index satisfies $G_t=i$. No independence between coordinates is needed.

Write $\|\cdot\|_2$ for the Euclidean norm on $\mathbb R^n$, and set $b=(b_1,\ldots,b_n)$ and $M(T)=(M_1(T),\ldots,M_n(T))$. Since $b_i\ge0$,
\[
 \|b\|_2^2=\sum_i b_i^2\le\left(\sum_i b_i\right)^2=(T/N)^2.
\]
Apply Cauchy--Schwarz pathwise, and then apply it again to the random variable $\|M(T)\|_2$, to obtain
\begin{align}
 \mathbb E\sum_iw_iM_i(T)
 &=-\mathbb E\sum_i b_iM_i(T)\notag\\
 &\ge-\frac TN\mathbb E\|M(T)\|_2
 \ge-\frac TN\left(\sum_i\mathbb E M_i(T)^2\right)^{1/2}\notag\\
 &\ge-\frac TB\sqrt{\mathbb E S_T}.\label{low:eq:weighted}
\end{align}
The first line uses $\mathbb E M_i(T)=0$. This estimate permits dependence between $b_i$ and the terminal martingale values and therefore applies to adaptive guesses.

%% file: sections/lower_bound.tex
\subsection{The quantitative lower bound and its asymptotic consequence}\label{sec:quantitative}
We now combine the conditional probability comparison and the martingale estimates to bound the minimum expectation. Keeping the final segment length $B$ as a parameter first gives an inequality for finite decks; a choice of $B$ then proves Theorem~\ref{thm:lower}. The key is to use $\E S_T\le\E S_N\le m$ only for a strategy attaining the minimum expectation, since this upper bound need not hold for an arbitrary strategy.
\begin{proposition}\label{low:thm:main}
For positive integers $m,n$ and an integer $m\le B<mn$, let $N=mn$ and $T=N-B$. Then
\begin{equation}\label{low:eq:finite}
 P^-_{m,n}\ge
 \frac{mT}{N}-\frac TB\sqrt m-\frac{m^2T}{NB}.
\end{equation}
\end{proposition}
\begin{proof}
Choose a deterministic strategy attaining $P^-_{m,n}$, whose existence was established in Section~\ref{sec:model}. Take expectations in~\eqref{low:eq:reconstruction} and apply~\eqref{low:eq:drift} and~\eqref{low:eq:weighted} to get
\begin{equation}\label{low:eq:self}
 \E S_T\ge\frac{mT}{N}
 -\frac TB\sqrt{\E S_T}-\frac{m^2T}{NB}.
\end{equation}
Always guessing one label has expectation $m$, so this minimizing strategy satisfies $\E S_T\le\E S_N=P^-_{m,n}\le m$. Replace $\sqrt{\E S_T}$ on the right of~\eqref{low:eq:self} by $\sqrt m$, and use $P^-_{m,n}\ge\E S_T$ on the left, to obtain~\eqref{low:eq:finite}.
\end{proof}
To see the role of the truncation length, take $0<\delta<1$ and assume $\delta n\ge1$ and $B=\lceil\delta N\rceil<N$. Then $B\ge m$, and
\[
 \frac{mT}{N}\ge m(1-\delta)-\frac1n,\qquad
 \frac TB\le\frac1\delta,\qquad
 \frac{m^2T}{NB}\le\frac{m}{n\delta}.
\]
Substituting into~\eqref{low:eq:finite} gives
\begin{equation}\label{low:eq:delta}
 P^-_{m,n}\ge m-m\delta-\frac{\sqrt m}{\delta}
 -\frac{m}{n\delta}-\frac1n.
\end{equation}
The four error terms in~\eqref{low:eq:delta} come from truncation, martingale fluctuations, the conditional probability comparison, and integer rounding, respectively. When $n\ge m$, the dominant terms are $m\delta$ and $\sqrt m/\delta$, which are equal at $\delta=m^{-1/4}$.
\begin{proof}[Proof of Theorem~\ref{thm:lower}]
If $1\le m<16$, then $m^{1/4}<2<3$ and hence $m-3m^{3/4}<0$. The lower bound follows directly from $P^-_{m,n}\ge0$. Suppose now that $m\ge16$. Set $\delta=m^{-1/4}$, so $0<\delta\le1/2$ and $\delta n\ge m^{3/4}\ge1$. Also $N=mn\ge256$, and therefore $\lceil\delta N\rceil<\delta N+1\le N/2+1<N$. Thus~\eqref{low:eq:delta} applies and gives
\[
 P^-_{m,n}\ge m-2m^{3/4}-\frac{m^{5/4}}n-\frac1n
 \ge m-2m^{3/4}-m^{1/4}-\frac1n.
\]
Since $m^{1/4}\le m^{3/4}/4$ and $1/n\le1/16\le m^{3/4}/4$, the last two error terms sum to at most $m^{3/4}$, proving the lower bound. The upper bound follows by always guessing one label. Finally,
\[
 0\le1-\frac{P^-_{m,n}}m\le3m^{-1/4}\longrightarrow0,
\]
and the right-hand side is independent of $n$, proving uniform convergence.
\end{proof}

%% file: sections/discussion.tex
\section{Further questions}
We have studied the minimum and maximum expected scores attainable with only yes/no feedback after each guess in a uniformly shuffled multiset deck. The two main conclusions are
\[
 \frac{P^-_{m,n}}{m}\xrightarrow[m\to\infty]{}1
 \quad\text{uniformly for $n\ge m$},
 \qquad
 \sup_{m,n\ge1}\frac{P^+_{m,n}}{m}=e-1.
\]
The first says that, at large multiplicity, even deliberately avoiding correct guesses cannot reduce the expectation by a fixed proportion. The second determines the best multiplicative constant valid for all decks. These answer the two questions posed by Diaconis, Graham, and Spiro~\cite[Conjectures~1 and~2]{DGS}. The proofs bound the extrema without characterizing optimal strategies. The principal remaining questions concern the exact extrema at fixed multiplicity and the true scale of the effect of feedback at large multiplicity.

\subsection{Fixed multiplicity: does the maximum expectation have a limit?}
Sharpness of the universal constant $e-1$ is determined by $m=1$, so it does not measure the loss caused by repeated cards themselves. For each fixed $m>1$, Theorem~\ref{thm:upper} gives a bounded sequence,
\[
 m\le P^+_{m,n}\le m\sum_{k=1}^{n}\frac1{k!}<(e-1)m,
\]
but boundedness does not ensure convergence.
\begin{question}\label{question:fixed-m}
For each fixed integer $m\ge2$, does the limit
\[
 \lim_{n\to\infty}P^+_{m,n}
\]
exist? If so, can its value be determined, particularly in the first repeated-card case $m=2$?
\end{question}
Even for $m=2$, the available bounds do not give an exact answer. The construction of Diaconis, Graham, and Spiro~\cite[Theorem~2]{DGS} and our upper bound give
\[
 2.91\le\liminf_{n\to\infty}P^+_{2,n}
 \le\limsup_{n\to\infty}P^+_{2,n}
 \le2(e-1).
\]
For general fixed $m$, the limiting subsequence length associated with the increasing strategy also gives a lower bound~\cite{CDHSY}. Our upper bound comes instead from another exactly solvable experiment: if the component deck at each position is revealed in advance, the optimal expectation is $mv_n$. Thus one may ask how much expected score is lost by hiding the deck labels. More precisely, does the nonnegative difference
\[
 mv_n-P^+_{m,n}
\]
converge as $n\to\infty$ for fixed $m>1$, and is its limit strictly positive? Since $mv_n\to(e-1)m$, convergence of this difference is equivalent to convergence in Question~\ref{question:fixed-m}. Strict positivity further distinguishes sharpness of the universal constant from the possibility of approaching it at every fixed multiplicity. This formulation connects the exact extremal problem to the sole additional information used in the proof: can a player without deck labels attain the optimum available when those labels are revealed?

\subsection{Large multiplicity: is the loss in the minimum expectation of square-root order?}
Theorem~\ref{thm:lower} determines the first-order term $m$, but not the scale of the loss achievable by avoiding correct guesses. Together with the construction of Diaconis, Graham, and Spiro~\cite[Theorem~1]{DGS}, it gives
\[
 \frac{\sqrt m}{40}
 \le m-P^-_{m,n}
 \le 3m^{3/4}
\]
when $n\ge8m$. The different exponents leave the second-order effect of feedback undetermined.
\begin{question}\label{question:sqrt}
Is there a constant $C>0$, independent of $m,n$, such that for all $n\ge m\ge1$,
\[
 m-P^-_{m,n}\le C\sqrt m\,?
\]
\end{question}
An affirmative answer would establish a loss of order $\Theta(\sqrt m)$ for $n\ge8m$. In the same range, the existing construction and Nie's upper bound already determine the gain in the maximum expectation to be of square-root order for sufficiently large $m$~\cite{DGS,Nie}. Question~\ref{question:sqrt} therefore asks whether seeking and avoiding correct guesses have the same second-order scale. Such symmetry does not follow directly from the definitions of the two strategy classes.

Equation~\eqref{low:eq:delta} also identifies a limitation of the present method. For its two indicated error terms alone,
\[
 m\delta+\frac{\sqrt m}{\delta}\ge2m^{3/4}
 \qquad(\delta>0).
\]
Changing only the truncation point cannot reduce this estimate to square-root order. One possible improvement concerns~\eqref{low:eq:weighted}: we use only the total mass of the terminal weights, without exploiting the distribution of guesses among labels or its joint constraints with the terminal martingale values. Another possibility is to retain more information from the final segment and reduce the loss from truncation. With two labels, yes/no feedback determines the actual card, and the loss does have square-root order; see Appendix~\ref{app:feedback}. For general decks, labels at past failed positions remain unknown. This is the distinction that requires further understanding.

%% file: sections/appendix_counts.tex
\section{From feedback histories to exact counts}\label{app:counts}
The main proof uses only a lower bound for a conditional probability. Exact optimization still does not require the entire feedback history: after successful positions are deleted, the multiplicities and failure counts for each label form a sufficient state. Diaconis, Graham, and Holmes~\cite[Section~4]{DGH} express the same conditional count as the permanent of a matrix with rectangular zero blocks. We give the count and the recursion for the optimal expectation directly in terms of multiset permutations.

After successful positions are deleted, let $r=(r_1,\ldots,r_n)$ be the multiplicities of cards at positions whose labels remain undetermined, and let $f=(f_1,\ldots,f_n)$ count the positions forbidding each label. Put $M=\sum_i r_i$ and $L=M-\sum_i f_i$. For two problems with the same multiplicities and forbidden-position counts, match the positions forbidding each label and then match the free positions. This gives a bijection between the allowed words. Their number therefore depends only on $r,f$; denote it by $C(r,f)$. The following is the multiset form of the inclusion--exclusion formula used for conditional card-guessing counts~\cite[Theorem~2 and equations~(8), (9)]{CDGM}. Write $(r)_k=r(r-1)\cdots(r-k+1)$, with $(r)_0=1$. In the sum below each $k_i$ is an integer, and $K=k_1+\cdots+k_n$.
\begin{proposition}
For nonnegative integer vectors $r,f$ with $\sum_i f_i\le M$,
\begin{equation}\label{app:eq:count}
 C(r,f)=\frac1{\prod_i r_i!}
 \sum_{\substack{0\le k_i\le\min\{r_i,f_i\}\\1\le i\le n}}
 (-1)^K(M-K)!\prod_{i=1}^n\binom{f_i}{k_i}(r_i)_{k_i}.
\end{equation}
If $C(r,f)>0$ and $L\ge1$, the probability that a free position contains label $i$ is
\[
 p_i(r,f)=\frac{C(r-e_i,f)}{C(r,f)}\quad(r_i>0),
 \qquad p_i(r,f)=0\quad(r_i=0),
\]
where $e_i$ has a 1 in coordinate $i$ and zeros elsewhere.
\end{proposition}
\begin{proof}
First distinguish cards of the same label. Choose $k_i$ of the $f_i$ positions forbidding $i$, and require those positions to violate their restrictions. There are $\binom{f_i}{k_i}$ choices of positions and $(r_i)_{k_i}$ ways to fill them with distinct cards labelled $i$. Once all $K$ positions are filled, the other cards can be arranged in $(M-K)!$ ways. Apply inclusion--exclusion to the violations and divide by the $\prod_i r_i!$ identity assignments corresponding to each word to obtain~\eqref{app:eq:count}. Finally, fixing label $i$ at a free position and deleting that position gives a bijection with allowed words of multiplicities $r-e_i$ and restriction counts $f$. The ratio of counts is the asserted probability.
\end{proof}
These probabilities also give an exact backward recursion. Let $V^+(r,f)$ and $V^-(r,f)$ be the maximum and minimum expected remaining scores from this state. Both are zero when $L=0$. For $L\ge1$,
\begin{align*}
 V^+(r,f)&=\max_{i\in[n]}
 \{p_i(r,f)[1+V^+(r-e_i,f)]
 +(1-p_i(r,f))V^+(r,f+e_i)\},\\
 V^-(r,f)&=\min_{i\in[n]}
 \{p_i(r,f)[1+V^-(r-e_i,f)]
 +(1-p_i(r,f))V^-(r,f+e_i)\}.
\end{align*}
Branches of probability zero are omitted, so values at unreachable states need not be defined. A success deletes the current position and decreases $r_i$; a failure changes the current free position into a position forbidding $i$. Either transition decreases $L$ by 1, so the recursion terminates at the boundary and includes every choice in the next round. This also proves that the state is sufficient, rather than merely convenient notation for marginal probabilities. The initial state $r=(m,\ldots,m)$, $f=0$ gives $P^\pm_{m,n}$.

For the same yes/no feedback model, Alimohammadi et al.~\cite[Section~4.2 and Appendix~C]{ADRS} estimate or compute these permanents to obtain the conditional probabilities of the next card, and hence to implement the greedy strategy that chooses a most likely label at each round. The recursion above also includes the remaining score after either feedback outcome and therefore optimizes the total score. Solving the conditional counting problem does not remove this distinction. The cards counted by $r_i$ may lie in future positions or may already have been removed after a failed guess. The counts $f_i$ must also be retained to determine the next conditional probability. The switching comparison in the main proof replaces these exact counts by a uniform inequality and thereby avoids solving over all states.

%% file: sections/appendix_feedback.tex
\section{Two solvable cases and a comparison of feedback models}\label{app:feedback}
The lower bound in the main theorem describes growing multiplicities and need not give useful numerical information for small ones. We first give the exact answer for $m=1$, then compute both extrema for two labels. The latter case also reveals the importance of the feedback model: with two labels, knowing whether a guess is correct is equivalent to knowing the card; with more labels, it is not.

\subsection{Each label occurs once}
When $m=1$, after guessing a label correctly, one can keep guessing it and never succeed again. Diaconis and Graham~\cite[Theorem~6]{DG} proved that a minimizing strategy guesses $1,2,\ldots,n$ in order until the first success, then continues guessing that successful label. If $d_n$ denotes the number of permutations of $[n]$ without fixed points, then
\[
 P^-_{1,n}=1-\frac{d_n}{n!}
 =1-\sum_{k=0}^n\frac{(-1)^k}{k!}
 \longrightarrow1-e^{-1}.
\]
Indeed, the score is either 0 or 1, and it is 0 exactly when the original permutation has no fixed points. Inclusion--exclusion gives $d_n=\sum_{k=0}^n(-1)^k\binom nk(n-k)!$, yielding the formula. Together with Lemma~\ref{up:lem:single}, this gives
\[
 P^-_{1,n}\longrightarrow1-e^{-1},\qquad
 P^+_{1,n}\longrightarrow e-1.
\]
There is no contradiction with $P^-_{m,n}/m\to1$: here $m=1$ is fixed and $n$ grows, whereas in the main theorem $m$ tends to infinity.

\subsection{Closed formulas for two labels}
With two labels, an incorrect guess determines the true label, so the remaining multiplicities are known after every round. The current guess does not affect the information available in the next round. Guessing a more numerous label therefore maximizes the expectation, and guessing a less numerous label minimizes it. This is the prediction problem for uniform balanced binary sequences studied by Blackwell and Hodges~\cite[Section~0]{BH}. The calculation below can also be viewed as the specialization of the complete-feedback formulas of Diaconis and Graham~\cite[Section~2]{DG} to two equally numerous labels.
\begin{proposition}\label{app:prop:two}
For every positive integer $m$,
\[
 P^\pm_{m,2}=m\pm\frac12\left(\frac{4^m}{\binom{2m}{m}}-1\right).
\]
In particular, $P^-_{m,2}+P^+_{m,2}=2m$, and
\[
 P^\pm_{m,2}=m\pm\left(\frac{\sqrt{\pi m}}2-\frac12+O(m^{-1/2})\right).
\]
\end{proposition}
\begin{proof}
Use the strategy that guesses a more numerous label each round, choosing either label in a tie. Let $J$ be the number of rounds in which the two remaining multiplicities are equal and nonzero. When they are unequal, a correct guess decreases their maximum by 1, and an incorrect guess leaves it unchanged. When they are equal, their maximum is unchanged regardless of the outcome. The maximum decreases from $m$ to 0, so exactly $m$ successes occur at unequal states. At each equal state the conditional success probability is $1/2$, giving
\[
 P^+_{m,2}=m+\frac12\E J.
\]
This decomposition by equal states also appears in the distributional analysis of two-colour card guessing~\cite[Lemmas~1 and~2]{KPtwo}.

When $2k$ cards remain, the probability of having $k$ of each label is
\[
 \frac{\binom{2k}{k}\binom{2m-2k}{m-k}}{\binom{2m}{m}}
 \qquad(1\le k\le m).
\]
The numerator counts binary arrangements in the last $2k$ and the first $2m-2k$ positions separately. Thus
\[
 \E J=\frac1{\binom{2m}{m}}
 \sum_{k=1}^m\binom{2k}{k}\binom{2m-2k}{m-k}
 =\frac{4^m}{\binom{2m}{m}}-1.
\]
The last equality follows by comparing coefficients of $z^m$ in the formal power series identity $(\sum_{k\ge0}\binom{2k}{k}z^k)^2=(1-4z)^{-1}$; the $k=0$ term is $\binom{2m}{m}$. This proves the formula for the maximum expectation.

On the same card sequence, let the minimizing strategy always choose the other label from the maximizing strategy, including at ties. Each strategy can reconstruct the complete card sequence from its own yes/no feedback, so this pairing is legal. Exactly one succeeds in each round, and their scores sum to $2m$, giving the minimum expectation. Finally, Stirling's formula gives $\binom{2m}{m}=4^m(\pi m)^{-1/2}(1+O(m^{-1}))$, yielding the asymptotic expansion.
\end{proof}
For example,
\[
 (P^-_{1,2},P^+_{1,2})=(1/2,3/2),\qquad
 (P^-_{2,2},P^+_{2,2})=(7/6,17/6).
\]
The correction is already of square-root order for two labels, but the equivalence of feedback models here does not extend to more labels.

\subsection{Complete feedback versus partial feedback}
With two labels, partial and complete feedback provide the same information. For general decks, the extrema satisfy only comparison inequalities. When $m=n$ grows, even the second-order scale of the maximum expectation differs.

For general $n$, let $C^-_{m,n},C^+_{m,n}$ be the minimum and maximum expectations when the actual card is revealed after each round. Complete feedback can simulate yes/no feedback, so
\[
 C^-_{m,n}\le P^-_{m,n}\le m\le P^+_{m,n}\le C^+_{m,n}.
\]
Let $R_i(\ell)$ be the number of cards labelled $i$ among the last $\ell$ cards. Under complete feedback these counts are known before the next guess. Since the current guess does not change future information,
\[
 C^-_{m,n}=\sum_{\ell=1}^{mn}\frac{\E\min_iR_i(\ell)}{\ell},\qquad
 C^+_{m,n}=\sum_{\ell=1}^{mn}\frac{\E\max_iR_i(\ell)}{\ell}.
\]
This explains why the complete-feedback problem reduces to extrema of the remaining multiplicities. He and Ottolini~\cite{HO} study these extrema through the birthday problem for sampling without replacement, and Ottolini and Steinerberger~\cite[Main Theorem]{OS} further analyze the correction when $m,n$ grow together. For example, they prove that, as $m=n\to\infty$,
\[
 C^+_{m,m}=m+\left(\frac\pi{\sqrt2}+o(1)\right)\sqrt{m\log m},
\]
whereas Nie's partial-feedback upper bound~\cite[Theorem~1.3]{Nie} has a correction of only $O(\sqrt m)$ in the same range. Thus the two feedback models already differ at the second-order scale.

For $m=1$, the contrast is more immediate: with $\ell$ distinct cards remaining, guessing any label known to remain succeeds with probability $1/\ell$, so $C^+_{1,n}=\sum_{\ell=1}^n1/\ell$. The maximum expectation under complete feedback diverges with $n$, whereas $P^+_{1,n}<e-1$ remains bounded. The quantity $m-s_i(t)$ in the main proof includes cards at past failed positions and cannot replace the actual remaining count $R_i(\ell)$ here.

%% file: bibliography.tex
\bibliographystyle{project-amsplain}
\bibliography{references}

%% file: main.bbl
\providecommand{\bysame}{\leavevmode\hbox to3em{\hrulefill}\thinspace}
\providecommand{\MR}{\relax\ifhmode\unskip\space\fi MR }
\providecommand{\MRhref}[2]{
  \href{http://www.ams.org/mathscinet-getitem?mr=#1}{#2}
}
\providecommand{\href}[2]{#2}
\begin{thebibliography}{10}

\bibitem{ADRS}
Y.~Alimohammadi, P.~Diaconis, M.~Roghani, and A.~Saberi, \emph{{Sequential
  importance sampling for estimating expectations over the space of perfect
  matchings}}, The Annals of Applied Probability \textbf{33} (2023), no.~2,
  999--1033.

\bibitem{BH}
D.~Blackwell and J.~L. Hodges, Jr., \emph{{Design for the control of selection
  bias}}, The Annals of Mathematical Statistics \textbf{28} (1957), no.~2,
  449--460.

\bibitem{CDGM}
F.~R.~K. Chung, P.~Diaconis, R.~L. Graham, and C.~L. Mallows, \emph{{On the
  permanents of complements of the direct sum of identity matrices}}, Advances
  in Applied Mathematics \textbf{2} (1981), no.~2, 121--137.

\bibitem{CDHSY}
A.~Clifton, B.~Deb, Y.~Huang, S.~Spiro, and S.~Yoo, \emph{{Continuously
  increasing subsequences of random multiset permutations}}, European Journal
  of Combinatorics \textbf{110} (2023), article no.~103708.

\bibitem{DG}
P.~Diaconis and R.~L. Graham, \emph{{The analysis of sequential experiments
  with feedback to subjects}}, The Annals of Statistics \textbf{9} (1981),
  no.~1, 3--23.

\bibitem{DGHS}
P.~Diaconis, R.~L. Graham, X.~He, and S.~Spiro, \emph{{Card guessing with
  partial feedback}}, Combinatorics, Probability and Computing \textbf{31}
  (2022), no.~1, 1--20.

\bibitem{DGH}
P.~Diaconis, R.~L. Graham, and S.~P. Holmes, \emph{{Statistical problems
  involving permutations with restricted positions}}, State of the Art in
  Probability and Statistics, Institute of Mathematical Statistics Lecture
  Notes--Monograph Series, vol.~36, Institute of Mathematical Statistics,
  Beachwood, OH, 2001, pp.~195--222.

\bibitem{DGS}
P.~Diaconis, R.~L. Graham, and S.~Spiro, \emph{{Guessing about guessing:
  practical strategies for card guessing with feedback}}, The American
  Mathematical Monthly \textbf{129} (2022), no.~7, 607--622.

\bibitem{HO}
J.~He and A.~Ottolini, \emph{{Card guessing and the birthday problem for
  sampling without replacement}}, The Annals of Applied Probability \textbf{33}
  (2023), no.~6B, 5208--5232.

\bibitem{KPtwo}
M.~Kuba and A.~Panholzer, \emph{{On card guessing with two types of cards}},
  Journal of Statistical Planning and Inference \textbf{232} (2024), article
  no.~106160.

\bibitem{Nie}
Z.~Nie, \emph{{The number of correct guesses with partial feedback}}, preprint,
  2022, arXiv:2212.08113.

\bibitem{OS}
A.~Ottolini and S.~Steinerberger, \emph{{Guessing cards with complete
  feedback}}, Advances in Applied Mathematics \textbf{150} (2023), article
  no.~102569.

\end{thebibliography}
